\documentclass{article}
\usepackage{amsmath,amsfonts,amssymb,amsthm,url}

\usepackage{graphicx,mathrsfs} 

\newtheorem{theorem}{Theorem}
\newtheorem{lemma}{Lemma}

\title{A Free Analog of Bobkov's Gaussian Isoperimetry Inequality}
\author{Dima Shlyakhtenko}
\date{\today}

\begin{document}
\maketitle

\begin{abstract}
    We prove a one-variable functional inequality which is the free probability analog of Bobkov's isoperimetry inequality. The inequality involves the $L^1$ norm of the difference quotient of a function $f$ and can be viewed as a non-local isoperimetric inequality. We also prove related inequalities for subsets of an interval as well as for subsets of roots of Hermite polynomials. This paper is also an experiment in AI-based exploration of free analogs of classical probability statements.
\end{abstract}

\section{Introduction.} \label{sec:notation}  Bobkov's inequality \cite{Bobkov1997} is a generalization of the Gaussian isoperimetric inequality. One formulation of the inequality is the following. Let $f:\mathbb{R}^2\to[0,1]$ be a sufficiently smooth function (e.g., $C^1$).  Let $\phi$ denote the Gaussian density $\phi(x)=(2\pi)^{-1/2} \exp(-x^2/2)$, and let $\psi$ be the inverse of the Gaussian quantile map determined by $$
\int_{-\infty}^{\psi(u)} \phi(t)dt = u,
$$  and consider the ``profile function'' $$I(u) = \phi(\psi(u)).$$  Let $\gamma_d (dx) = (2\pi)^{-d/2} \exp(-\Vert x\Vert^2/2)$ be the Gaussian measure on $\mathbb{R}^d$.  Then Bobkov's inequality states that
\begin{equation}\label{eq:strongBobkov}
I\left(\int fd\gamma\right)  \leq \int \sqrt{ (I\circ f)^2 + \Vert \nabla f\Vert^2} d\gamma.
\end{equation}
or, in an equilvaent form (see \cite{BartheMaurey}),
\begin{equation}\label{eq:weakBobkov}
I\left(\int fd\gamma\right) - \int I\circ f d\gamma \leq \int \Vert \nabla f\Vert d\gamma.
\end{equation}

The key feature of Bobkov's inequality is the right-hand side of \eqref{eq:weakBobkov}, which involves the $L^1$ norm of the gradient of $f$. 

Bobkov's inequality has received much attention. For exmaple, Bakry and Ledoux proved \eqref{eq:strongBobkov} using the Ornstein-Uhlenbeck semigroup, and generalized the inequality to log-concave measures \cite{BakryLedoux1996}.  
Other proofs and generalizations, using Brownian motion theory, were given by Capitaine, Hsu, and Ledoux in \cite{CapitaineHsuLedoux} and  Barthe and Maurey \cite{BartheMaurey}.

In this paper we prove an analog of Bobkov's inequality \eqref{eq:weakBobkov} in free probability theory \cite{VoiculescuDykemaNica,NicaSpeicher} in the one-variable case  $d=1$. More precisely:
\begin{theorem} \label{thrm:FBI}
    Let $f \in C^1([-1,1],\sigma)$  and assume that $0\leq f\leq 1$.  Then \begin{equation}
    \label{eq:FBI}
        J\left(\int f d\sigma \right) - \int (J\circ f) d\sigma (x) \leq \int_{-1}^1 \int_{-1}^1 \left| \frac{f(x)-f(y)}{x-y}\right| d\sigma(x) d\sigma(y).
    \end{equation}
\end{theorem}

Here $\sigma$ is the semicircle law $\sigma(dx)=\frac{2}{\pi} 1_{[-1,1]}(x) \sqrt{1-x^2} dx$, and $J$ is the free profile map given by $$
J(x) = \frac{8}{3\pi}(1-\alpha^2(x))^{3/2}
$$
where $\alpha(x)$ is the inverse quantile map of the semicircle law determined by 
$$
\int_{-1}^{\alpha(x)} d\sigma(t) = x.
$$

Curiously, $J$ has a structure similar to $I$, except that it is (up to a constant) the composition of  the \emph{cube} of the semicircle density with the inverse quantile map.  On the other hand, the right-hand side of \eqref{eq:FBI} involves the $L^1$ norm of the difference quotient derivation, which is the free analog of the gradient (see e.g. \cite{Voiculescu1998}).  

By continuity and semicontinuity, Theorem~\ref{thrm:FBI} holds true for any bounded measurable function $f:[-1,1]\to[0,1]$.  If we consider the special case of $f=1_{X}$ where $X\subset [-1,1]$ is a measurable subset, \eqref{eq:FBI} simplifies. Indeed, $J(0)=J(1)=0$ and so $J\circ f=0$ in this case.  If we introduce the notation $$E(X)=\int_{x\in X} \int_{y\in [-1,1]\setminus X}  \frac{1}{|x-y|} d\sigma(x) d\sigma(y),$$ Theorem~\ref{thrm:FBI} immediately gives a Lieb-style non-local isoperimetry inequality:

\begin{theorem} \label{thrm:optimalSet}
   Let $X\subset[-1,1]$ be a Lebesgue-measurable set.  Then \begin{equation}\label{eq:FBIS}
    2 E(X) \geq J(\sigma(X)) = 2E([-1,\alpha(\sigma(X))]).
   \end{equation} 
\end{theorem}

In fact, Theorem~\ref{thrm:FBI} is equivalent to Theorem~\ref{thrm:optimalSet}, and we prove the latter theorem first. 

Note that the equality $J(\sigma(X)) = 2E([-1,\alpha(\sigma(X))])$ in Theorem~\ref{thrm:optimalSet} shows that the inequality  \eqref{eq:FBI} is tight, with equality achieved when $f$ is a characteristic function of a set of the form $[-1,a]$ or $[a,1]$. Thus $J$ is the optimal profile for our problem.

It is natural to ask whether multi-variable or operator-valued generalizations of the inequality exist, but we do not have any results in this direction.

It is also natural to ask whether the inequality \eqref{eq:FBIS} arises in some large-$n$ limit. We show that this is indeed the case: in fact, in the style of finite free probability \cite{MarcusSpeilmanSrivastava}, a similar inequality is true for roots of Hermite polynomials. More precisely, denoting by $R$ the set of roots of the $n$-th Hermite polynomial, we can define for any $A\subset R$ its energy $$E(A)=\sum_{x\in A}\sum_{y\in R\setminus A} \frac{1}{|x-y|}.$$  The statement is then that among all sets $A$ of fixed cardinality, $E(A)$ is minimized by an interval consisting of the first $|A|$ roots (see Theorem~\ref{thrm:HermiteIntervalOptimal} for the precise statement). 

The idea of the proof comes from exploring a monotonicity property of the quantile maps associated with rescaled semicircular measures (these are no longer probability measures), given by Lemma~\ref{lem:dIsMonotone}, or the analogous statement for the Hermite roots, see Lemma~\ref{lem:hermiteFacts}.\ref{fact:diMonotone}.  This monotonicity allows us to compare the energy of a root configuration of a polynomial $H_n$ with a similar configuration of roots, but of the polynomial $H_{n-1}$ (with no rescaling) and then proceed by induction until encountering a configuration that consists of all of the roots of $H_k$ for some $k<n$. 

This paper arose out of the author's curiosity as to whether it would be possible to set up an AI tool that would be able to automatically explore the landscape of free probability by formulating and subsequently proving or disproving free analogs of classical probability statements, and checking the proofs through formalization in a proof assistant. 

A tool like that remains a (distant?) dream. Nonetheless, this paper represents a full run through such a cycle, although with a lot of human assistance. ChatGPT 5.4 Pro and 5.5 Pro were essential tools in obtaining this result, by filling in proofs, providing numerical verification for conjecture --- or refuting them either numerically by through reasoning. In fact, the inequality of the paper was  (after some false starts) conjectured by ChatGPT.  However, many of the high-level ideas of the proof, including the use of Hermite roots and comparison between roots of an Hermite polynomial and its derivative (which led us to find the proof), were human-supplied, but elaborated by AI.  

Thus this paper represents an truly AI-assisted mathematical result --- many ideas were supplied and elaborated by the human author --- but it would not have been written without AI's involvement.

A complete formalization of this paper in Lean using Mathlib 4.30.0 was constructed using the Codex coding assistant supplied with a detailed proof outline written by ChatGPT. Although human assistance was required, it was minimal. The formalization is available on GitHub \cite{githublink}\footnote{The formalization differs from the current paper in that it derives Theorem~\ref{thrm:FBI} from Theorem~\ref{thrm:HermiteIntervalOptimal} and has a formally weaker form of Theorem~\ref{thrm:optimalSet}: it is stated for sets which can be well-approximated by Hermite roots. The statement is enough to derive Theorem~\ref{thrm:FBI}, which could in principle be used to bootstrap the full Theorem~\ref{thrm:optimalSet}, but we chose to forego the exercise. The formalization uses no axioms beyond the standard Lean axioms.}.

The formalization of Theorem~\ref{thrm:HermiteIntervalOptimal} was obtained in about a day after the proof was obtained, i.e., the time required was of the same order as needed for human verification (the full formalization took approximately a week, due to difficulties in handling the passage to the limit of Hermite roots). This demonstrated the emerging capability of AI tools to rapidly produce formally verified proofs of novel mathematical results.   The resulting AI-produced formalization serves its purpose, but the resulting codebase is not optimized for either length, clarity, readability, or reusability, and has many shortcomings.

\subsection{Acknowledgments}  The author is grateful to Djalil Chafai, Wilfrid Gangbo, John Hopper, and Terry Tao for a number of fruitful discussions. This paper would not have been possible without ChatGPT.  I am also grateful to Robert Shlyakhtenko for his advice on the formalization of this paper.

This research was sponsored in part by the Army Research Office and was accomplished under Grant
Number W911NF-25-1-0075. The views and conclusions contained in this document are those of the authors and
should not be interpreted as representing the official policies, either expressed or implied, of the Army Research
Office or the U.S. Government. The U.S. Government is authorized to reproduce and distribute reprints for
Government purposes notwithstanding any copyright notation herein.
Research was also supported in part by NSF grant DMS-2348633.

\section{Proof of Theorem~\ref{thrm:optimalSet}.}\label{sec:proofoFtheoremOptimalSet}
We now give a proof of Theorem~\ref{thrm:optimalSet}, which is essentially restated (with assumptions on the set $X$) as  Lemma~\ref{lem:FBIIntervals}. We could have derived that Lemma as an $n\to\infty$ limit of Theorem~\ref{thrm:HermiteIntervalOptimal}, but we prefer to give a direct argument that actually parallels its proof.

Let $$G(z) = \int \frac{1}{z-x} d\sigma(x) = 2 ( z - \sqrt{z^2 -1})$$ be the Cauchy transform of the semicircle law. Recall also that the Hilbert transform of the semicircle density $\rho(x) = (2/\pi) \sqrt{1-x^2}$ satisfies
$$
\operatorname{P.V.} \int \frac{d\sigma(y)}{x-y} = 2x.
$$
Equivalently, substituting $y=\alpha(u)$, $$
\operatorname{P.V.} \int \frac{1}{\alpha(u)-\alpha(v)}dv = 2 \alpha(u).$$

Fix $c<1$. Denote by $D_{\sqrt{c}} : [-1,1] \ni 
\lambda\mapsto \sqrt{c}\lambda 
\in [-\sqrt{c},\sqrt{c}]$ 
the homothety by $\sqrt{c}$.  Let $$\sigma_c = c D_{\sqrt{c}} \sigma.$$  
Then $\sigma_c$ is a finite measure of total mass $c$ with support $[-\sqrt{c},\sqrt{c}]$ and density $\rho_c(x) = (2/\pi) \sqrt{c-x^2}$.  The quantile map $\alpha_c : [0,c]\to[-\sqrt{c},\sqrt{c}]$ determined by $$\int_{-\sqrt{c}}^{\alpha_c(u)} d\sigma_c(t) = u,$$ is given by: $$
\alpha_c(u) = \sqrt{c}\alpha(u/c),$$ and the Cauchy transform of $\sigma_c$ is $G_c(z) = \sqrt{c} G(z/\sqrt{c})$. Moreover, 
\begin{equation} \label{eq:StieltjesForSigma_c}\operatorname{P.V.} \int_{0}^c \frac{dv}{\alpha_c(u)-\alpha_c(v)} =2\alpha_c(u).\end{equation}

\begin{lemma} \label{lem:dIsMonotone}
    For any $0<u<c'<c\leq 1$, the map $u\mapsto  \alpha_{c'}(u) - \alpha_c(u) $ is monotone increasing.
\end{lemma}
\begin{proof}
Note that $$\alpha_{c'}(u)-\alpha_c(u) = \sqrt{c'}\alpha(u/c') - \sqrt{c}\alpha(u/c) = \sqrt{c}\left(\sqrt{r}\alpha\left(w/r\right)  -\alpha(w)\right),$$ where we put $w=u/c$ and $r=c'/c$. Thus the statement of the Lemma is equivalent to the statement that for any $0<r<1$,  $$f(w) = 
\sqrt{r} \alpha(w/r) - \alpha(w)$$ is monotone increasing in $w$.
Differentiating $$\int_{-1}^{\alpha(u)} d\sigma(t) = u,$$ and letting $\rho(t)=(2/\pi)\sqrt{1-t^2}$, we get  $$
\alpha'(u) = \frac{1}{\rho(\alpha(u))}.$$ Hence
$$\frac{\pi}{2} f'(w) =\frac{1}{\sqrt{r}\sqrt{1-\alpha^2(w/r)}} - \frac{1}{\sqrt{1-\alpha^2(w)}}.$$ Therefore, it's sufficient to prove that $f'\geq 0$, i.e., $$
\sqrt{1-\alpha^2(w)} \geq \sqrt{r}\sqrt{1-\alpha^2(w/r)}.
$$
Let's divide both sides by $\sqrt{w}>0$ to get the equivalent inequality
$$
\frac{\sqrt{1-\alpha^2(w)}}{\sqrt{w}} \geq 
\frac{\sqrt{1-\alpha^2(w/r)}}{\sqrt{w/r}} = \frac{\sqrt{1-\alpha^2(w')}}{\sqrt{w'}},
$$ where we put $w'=w/r>w$.

Thus it's sufficient to prove that the function $$w\mapsto\frac{\sqrt{1-\alpha^2(w)}}{\sqrt{w}},$$ or, equivalently,
$$w\mapsto \frac{1-\alpha^2(w)}{w},$$ is monotone decreasing for $0<w<1$. 

Let's substitute $w=\alpha^{-1}(x)$.  Since $\alpha$ (and thus $\alpha^{-1}$) is monotone increasing, it's therefore sufficient to prove that 
$$g: x\mapsto \frac{1-x^2}{\alpha^{-1}(x)}$$ is monotone decreasing. We now compute the derivative of $g$, using $(\alpha^{-1})'(x)=\rho(x) = (2/\pi) \sqrt{1-x^2}$:
$$
g'(x) = \frac{1}{(\alpha^{-1}(x))^2 }\left(
-2x \alpha^{-1}(x) -  \frac{2}{\pi} (1-x^2)^{3/2}
\right).
$$
Thus it's enough to show that $g'(x)\leq 0$, i.e., $$
-2x\alpha^{-1}(x) \leq \frac{2}{\pi} (1-x^2)^{3/2}.$$ Since $\alpha^{-1}(x)\geq 0$, the inequality is clearly true if $x\geq  0$.  If $x\leq 0$, let $y=-x$, so that we need to show
$$
2y\alpha^{-1}(-y) \leq  \frac{2}{\pi} (1-y^2)^{3/2}.
$$
By definition, $$\alpha^{-1}(x)=\int_{-1}^{x} \frac{2}{\pi}\sqrt{1-t^2} dt = \frac{1}{\pi} \left(\arccos(-x) + x\sqrt{1-x^2}\right), $$ 
so $$2y\alpha^{-1}(-y) = \frac{2y}{\pi} \left( \arccos(y) - y\sqrt{1-y^2}\right),$$ and we need to show that
$${y} \left( \arccos(y) - y\sqrt{1-y^2}\right) 
 - (1-y^2)^{3/2}\leq 0
$$  Let's substitute $y=\cos\theta$; since $y\in [0,1)$, $0<\theta\leq \pi/2$.  The inequality becomes
$$
\cos\theta (\theta - \cos\theta \sin\theta ) - \sin^{3}\theta\leq 0.
$$
Simplifying, we get
\begin{align*}
\cos\theta &(\theta - \cos\theta \sin\theta ) - \sin^{3}\theta 
 = \theta \cos\theta -\cos^2\theta \sin\theta -\sin^3\theta \\
& = \theta \cos\theta -\sin\theta(\cos^2\theta +\sin^2\theta) 
= \theta\cos\theta -\sin\theta \\ & = \cos\theta (\theta -\tan\theta).
\end{align*}
Thus it's sufficient to prove that $h(\theta) = \theta - \tan\theta$ is non-positive for $0\leq \theta <\pi/2$. To this end, note that $h(0)=0$ and $h'(\theta)=1-\sec^2\theta<0$.
\end{proof}

For $A\subset[0,c]$ define 
$$E_c (A) = \iint_{\substack{u\in A\\v\in [0,c]\setminus A}} \frac{du dv}{|\alpha_c(u)-\alpha_c(v)|}, \quad B_c(A) = \iint_{\substack{u\in A,\\v\in [0,c]\setminus A,\\ v<u}} \frac{du dv}{\alpha_c(u)-\alpha_c(v)}.$$

\begin{lemma} \label{lem:EnergyInTermsOfInversion}
    With the above notation, we have:\\
    (i) $\displaystyle E_c(A) = -2 \int_A \alpha_c(u) du + 2 B_c (A).$\\
    (ii) $\displaystyle E_c([0,|A|]) = -2 \int_0^{|A|}\alpha_c(u) du$, where $|A|$ is the Lebesgue measure of $A$.
\end{lemma}
\begin{proof}
Let $A'=[0,c]\setminus A$. 
    By symmetry, $$\int_{A} \int_{A} \frac{dudv}{\alpha_c(u)-\alpha_c(v)}=0.$$ Note that $\alpha$ is monotone, so $\alpha(u)-\alpha(v)>0$ if $v<u$ and $\alpha(u)-\alpha(v)<0$ if $u<v$.  
    
    We compute:
    \begin{eqnarray*}
        E_c(A)&=& \int_{A}\int_{A'} \frac{dudv}{|\alpha_c(u)-\alpha_c(v)|} \\
        &=& \int_{u\in A}\int_{v\in A', v<u} \frac{dudv}{\alpha_c(u)-\alpha_c(v)} - \int_{u\in A}\int_{v\in A', u<v} \frac{dudv}{\alpha_c(u)-\alpha_c(v)} \\
        &=& 2 B_c(A) - \int_{u\in A} \int_{v\in A'} \frac{du dv}{\alpha_c(u)-\alpha_c(v)}  - \int_{u\in A}\int_{v\in A} \frac{du dv}{\alpha_c(u)-\alpha_c(v)} \\
        &=& 2B_c(A) - \int_{u\in A}\int_{0}^c \frac{dv}{\alpha_c(u)-\alpha_c(v)} du \\
        &=& 2B_c(A) - \int_{u\in A} 2\alpha_c(u) du
    \end{eqnarray*}

    This proves (i). For (ii), note that $B_c([0,a])=0$ for any $a$, since $v<u$ never happens if $u\in [0,a]$ and $v
    \in [0,a]'=[a,c]$.
\end{proof}

Let now $$D_c(A) = E_c(A) - E_c( [0,|A|]).$$  Our aim will be to prove that we always have $D_c(A)\geq 0$, showing that of all sets with the same measure as $A$, the interval $[0,|A|]$ minimizes energy.

\begin{lemma} \label{lem:canShrinkOneEnd}
    Let $0<c'<c\leq 1$, and let $A$ be a subset of $[0,c']$.  Then $D_c(A)\geq D_{c'}(A)$.
\end{lemma}
\begin{proof}
    Let $\delta(u) = \alpha_{c'} (u)  - \alpha_c(u)$.   By Lemma~\ref{lem:EnergyInTermsOfInversion}, \begin{eqnarray*}
    D_c(A) &=&  -2 \int_{A} \alpha_c(u)du + 2B_c(A) + 2 \int_{0}^{|A|} \alpha_{c}(u) du   \\
        & & + 2\int_{A} \alpha_{c'} du - 2B_{c'} (A) - 2\int_{0}^{|A|} \alpha_{c'}(u) du \\
        &=& 2 (B_{c}(A)-B_{c'}(A)) + 2\left( \int_A \delta(u) du - \int_{0}^{|A|} \delta(u) du\right).
    \end{eqnarray*}
    Then by Lemma~\ref{lem:dIsMonotone}, $\delta(u)$ is monotone increasing in $u$. 
    Since $\delta(u)$ is increasing and the sets $A$ and $[0,|A|]$ have the same Lebesgue measure, $$  \int_A \delta(u) du - \int_{0}^{|A|} \delta(u) du \geq 0,
    $$ and thus 
    \begin{eqnarray*}
        D_c(A)&\geq&  2 (B_{c}(A)-B_{c'}(A)) \\
        &=& 2 \iint_{\substack{u\in A,\\v\in [0,c]\setminus A,\\ v<u}}\left[ \frac{1}{\alpha_c(u)-\alpha_c(v)}  - \frac{1}{\alpha_{c'}(u)-\alpha_{c'}(v)}  \right]dudv.
    \end{eqnarray*}
    If  $v<u$, $\delta(u)-\delta(v)\geq 0$, and so 
    $$\alpha_{c'}(u) - \alpha_{c'}(v)  = \alpha_c(u) -\alpha_c(v) + \delta(u) - \delta (v) \geq \alpha_c(u)-\alpha_c(v).$$ Thus 
    $$
        \frac{1}{\alpha_c(u)-\alpha_c(v)}  - \frac{1}{\alpha_{c'}(u)-\alpha_{c'}(v)} \geq 0.  
    $$
    It follows that $$
    D_c(A) \geq 2\iint_{\substack{u\in A,\\v\in [0,c]\setminus A,\\ v<u}}\left[ \frac{1}{\alpha_c(u)-\alpha_c(v)}  - \frac{1}{\alpha_{c'}(u)-\alpha_{c'}(v)}  \right]dudv\geq 0,
    $$
    as claimed.
\end{proof}

We now note that by symmetry of $\rho_c$,  $\alpha_c$ also satisfies a symmetry condition: $$\alpha_c(c-u) = -\alpha_c(u).$$  Thus if $A\subset [0,c]$ and $A' = [0,c]\setminus A$, and $A''$ is the image of $A$ under the map $u\mapsto c-u$, then $$
    E_c(A)=E_c(A') = E_c(A''). 
$$ In particular, $E_c([0,|A|]) = E_{c}([|A|,c]) = E_c([0,c-|A|])$ and therefore 
\begin{equation}
    D_c(A) =  D_c(A') \label{eq:DcVsComplement}
\end{equation}

\begin{lemma} \label{lem:FBIIntervalsUnionUCoords}
    Let $A\subset[0,c]$ be a finite union of intervals. Then
    \[
    D_c(A)\geq 0
    \]
\end{lemma}

\begin{proof}
    The proof is by induction on the number $m$ of component intervals of $A$.  

    Suppose first $m=1$, so that $A = [a_1,a_2]$.  Applying Lemma~\ref{lem:canShrinkOneEnd} with $c'=a_2$ gives us $$D_{c}(A) \geq D_{a_2}([a_1,a_2]).$$  Applying the \eqref{eq:DcVsComplement} with $c=a_2$ shows that 
    $$D_{a_2}([a_1,a_2])=D_{a_2}([0,a_2-a_1])=0,$$ since $[0,a_2-a_1] = [0,|[0,a_2-a_1]|]$.  

    Suppose now the statement is true for any $A$ consisting of at most $m$ intervals, and assume that $A=A_0 \sqcup [a_1,a_2]$ with $A_0$ a set consisting of $m$ intervals. 
    
    Applying  Lemma~\ref{lem:canShrinkOneEnd} with $c'=a_2$ gives us $$D_{c}(A) \geq D_{a_2}(A_0 \sqcup [a_1,a_2]).$$ 
    Let $A' = [0,a_2]\setminus A$ be the complement of $A$ in the set $[0,a_2]$. By \eqref{eq:DcVsComplement}, $D_c(A) \geq D_{a_2}(A) = D_{a_2}(A')$.
    Note that $A'\subset [0,a_1]$ and thus we can apply the same argument as before to conclude that $D_{c} (A)\geq D_{a_2}(A')\geq D_{a_1}(A')$.  Now $A'\subset[0,a_1]$ consists of at most $m+1$ intervals, but is such that both $0$ and $a_1$ are in some interval of $A'$.  Thus if we denote by $A''=[0,a_1]\setminus A'$ the complement of $A'$, then $A''$ consists of at most $m$ intervals, and so $D_{c} (A) \geq D_{a_2}(A'')\geq 0$ by induction. 
\end{proof}

This gives us an immediate corollary: 

\begin{lemma} \label{lem:FBIIntervals}
    Let $X\subset[-1,1]$ be a finite union of intervals. Then
    \[
   J(\sigma(X))\leq 2 E(X).
    \]
\end{lemma}
\begin{proof}
    We apply Lemma~\ref{lem:FBIIntervalsUnionUCoords} to the set $A = \alpha^{-1}(X)$ to conclude that 
    \begin{eqnarray*}
        E(X) &=& \iint_{x\in X,y\in [-1,1]\setminus X} \frac{d\sigma(x)d\sigma(y)}{|x-y|}  \\
        &=& \iint_{u\in A,v\in[0,1]\setminus A}\frac{dudv }{|\alpha(u)-\alpha(v)|}  \\ 
        &=& E_1(A) \\ 
        &\geq& E_1([0,|A|]) 
    \end{eqnarray*}
    Let $$f(a) = E_1([0,a]) = 
        \int_{0}^{a} \int_{a}^1 \frac{dudv}{\alpha(u)-\alpha(v)}.$$
    Then 
    \begin{eqnarray*}f'(a) &=& \int_{0}^a \frac{du}{\alpha(u)-\alpha(a)}  -\int_{a}^1 \frac{dv}{\alpha(a)-\alpha(v) } \\&=& 
    - \operatorname{P.V.} \int_{0}^1 \frac{dv}{\alpha(a)-\alpha(v)} = -2 \alpha(a), \end{eqnarray*} where for the last equality we used \eqref{eq:StieltjesForSigma_c}.

    Let $x=\alpha(a)$, so that $$\frac{da}{dx} = \rho_1(x) = \frac{2}{\pi}{\sqrt{1-x^2}}.$$ Thus
    $$\frac{df}{dx} = -2\alpha(x) \rho(x) = -\frac{4}{\pi} x\sqrt{1-x^2}.$$  Since $f(0)=0$, we get $$f(x) = \frac{4}{3\pi} (1-x^2)^{3/2}$$ so $$f(a) = \frac{4}{3\pi} (1-\alpha(a)^2)^{3/2} = \frac{1}{2}J(a). $$

    In particular, $E_1([0,|A|]) = f(|A|)$ with $|A| = |\alpha^{-1}(X)| =\sigma(X)$, and thus 
    $$2E(X)\geq J(\sigma(X)),$$ as claimed.
\end{proof}

Theorem~\ref{thrm:optimalSet} now follows by a standard argument approximating an arbitrary measurable set by unions of intervals, and  Lemma~\ref{lem:FBIIntervals}.

\section{Proof of Theorem~\ref{thrm:FBI}}
Recall that $J(x)=\frac{8}{3\pi} (1-\alpha(x)^2)^{3/2}$, where $\alpha(x)$ satisfies $$
\int_{-1}^{\alpha(x)} \frac{2}{\pi} \sqrt{1-t^2} dt = x.$$ Differentiating this gives $$\frac{2}{\pi} \sqrt{1-\alpha(x)^2} \alpha'(x)=1.$$ 
Thus
$$
J'(x) = -\frac{8}{\pi} \alpha(x)\sqrt{1-\alpha(x)^2} \alpha'(x) = -4\alpha(x).
$$
In particular, $$J''(x) = -4\alpha'(x) <0 $$ since $\alpha$ is the inverse of $x\mapsto \int_{-1}^x d\sigma(t)$, which is a monotone-increasing function. In particular, $J$ is a concave function.  

Let $h (x) = - 1/J''(x) = 1/(4\alpha'(x)) = \frac{1}{2\pi} {\sqrt{1-\alpha(x)^2}}$.  
Computing gives $$h''(x)=-\frac{\pi(1+\alpha(x)^2)}{8(1-\alpha(x)^2)^{5/2}}<0,$$ so $h$ is also concave.

Let now $f:[-1,1]\to[0,1]$ be a measurable function. 
Let $$\mathscr{E}(f) = \iint \left|\frac{f(x)-f(y)}{x-y}\right| d\sigma(x) d\sigma(y)$$ (interpreted as $+\infty$ if the integral diverges). 

For each $t\in [0,1]$, let $$A_t = \{x : f(x)> t\} $$ be a measurable subset of $[-1,1]$.  Then 
$$
f(x) =\int_{0}^1 1_{A_t} (x) dt, \qquad |f(x)-f(y)|=\int_{0}^1 |1_{A_t} (x) - 1_{A_t} (y)| dt.
$$

Thus $$\mathscr{E}(f) = \int_{0}^1\iint  \left| \frac{1_{A_t} (x) - 1_{A_t} (y)}{x-y}\right| d\sigma(x)d\sigma(y) dt = 
\int_0^1 2E(A_t) dt.$$

Applying the layer-cake lemma to $J\left(\int fd\sigma)\right) - \int J\circ f d\sigma $ (which is possible since $J$ is concave and $1/J''$ is concave) gives the inequality
$$J\left(\int fd\sigma\right) - \int J\circ f d\sigma \leq \int_0^1 J(\sigma(A_t)) dt. $$

Applying Theorem~\ref{thrm:optimalSet} to $A_t$ for each $t$ gives us $$ J(\sigma(A_t)) dt  \leq \int_0^1 2E(A_t) dt = \mathscr{E}(f).$$ Putting this together gives Theorem~\ref{thrm:FBI}.

\section{An analog of Theorem~\ref{thrm:optimalSet} for roots of Hermite polynomials.}
Let $H_n$ be the $n$-th Hermite polynomial \cite[\S5.5]{Szego} determined recursively by:
$$H_0(x)=1,\qquad H_1(x) = 2x,$$
$$H_{n+1}(x) = 2x H_n(x) - 2n H_{n-1}(x),\qquad n\geq 1.$$
Then $H'(x) = 2n H_{n-1}(x)$ and these polynomials satisfy the differential equation $$
H_n''(x) - 2x H'_n(x) + 2nH_n(x) = 0.$$  It is well known \cite[\S3.3 and Chapter 5]{Szego} that if the roots $\xi_{n,1},\dots,\xi_{n,n}$ of these polynomials are scaled by the factor $(2n)^{-1/2}$, then they lie in the interval $[-1,1]$. Moreover, the atomic probability measures $\sigma_n = (1/n)\sum_j \delta_{\xi_{n,j}}$ converge weakly as $n\to\infty$ to the semicircle law. In this section, we prove an analog of Theorem~\ref{thrm:optimalSet} for subsets of the set of roots of $H_n$. This result is presented in Theorem~\ref{thrm:HermiteIntervalOptimal} and can in fact be used to derive Theorem~\ref{thrm:HermiteIntervalOptimal} by a limit argument.

We first gather some standard facts about the polynomials $H_n$ needed for our proof, and summarize them as a Lemma. We could not find a precise reference for statement (\ref{fact:diMonotone}), so we give a direct argument loosely based on \cite{IsmailMuldoon1991}.  We refer the reader to \cite[\S3.3, \S3.4]{Szego}  for the rest of the listed properties of the roots. 

\begin{lemma} The following properties hold: \label{lem:hermiteFacts}
\begin{enumerate}\renewcommand{\theenumi}{\alph{enumi}}
    \item Each $H_n(x)$ has $n$ simple real roots $\xi_{n,1}<\xi_{n,2}<\dots<\xi_{n,n}$ 
    \item The set of roots is symmetric about zero.
    \item The roots of $H_{n+1}$ and $H_n$ interlace: $$ \xi_{n+1,1} < \xi_{n,1} <\xi_{n+1,2}<\xi_{n,2} <\cdots <\xi_{n,n} < \xi_{n+1,n+1}. $$
    \item For each $1\leq i\leq n$,\ \  $\displaystyle \sum_{j\neq i} \frac{1}{\xi_{n,i}-\xi_{n,j}} = \xi_{n,i}.$ \label{fact:xiRootStietjes}
    
    \item \label{fact:rootStietjes}Let $x_{n,i} = {(2n)}^{-1/2}\xi_{n,i}$ be rescaled roots of $H_n$. Then $x_{n,i} \in [-1,1]$, and each $1\leq i\leq n$, $$\sum_{j\neq i} \frac{1}{x_{n,i}-x_{n,j}} = 2n \ x_{n,i}.$$ 
    \item \label{fact:diMonotone} Let $d_i = x_{n-1,i}-x_{n,i}$. Then $$d_1<d_2<\cdots<d_{n-1}.$$
    \end{enumerate}
\end{lemma}

\begin{proof}[Proof of (\ref{fact:diMonotone})]
    Fix $n$ and $1\leq i\leq n-2$. 

    Let $$R(x) = \frac{H'_n(x)}{H_n(x)} = \sum_{j=1}^n \frac{1}{x-\xi_{n,j}}.$$  The function $R(x)$ satisfies the Riccati equation \begin{equation}\label{eqn:RiccatiForR} R'(x) = -R(x)^2 + 2xR(x) - 2n.\end{equation}
    
    For $0<t< \xi_{n,i+1}-\xi_{n,i}$, let $$ Z_i(t) = R(\xi_{n,i} + t).$$
    Then $Z_i(t)$ is a smooth function of $t$ on $(0,\xi_{n,i+1}-\xi_{n,i})$.  
    
    Since the roots of $H_n$ are simple,  $Z_i(t)\to +\infty$ as $t\downarrow 0$ and $Z_i(t)\to -\infty$ as $t\uparrow \xi_{n,i+1}-\xi_{n,i}$. Thus $Z_i(t)$ has a zero in the interval $(0,\xi_{n,i+1}-\xi_{n,i})$.  If $Z_i(t)=0$, then $$ 0 = R(\xi_{n,i}+t)= \frac{H_n'(\xi_{n,i}+t)}{H_n(\xi_{n,i}+t)},$$ so that $H_n'(\xi_{n,i}+t)=0$. Since the zeros of $H_n$ and $H_n'$ interlace, this means that $\xi_{n,i}+t=\xi_{n-1,i}$. 
    
    Thus the unique zero of $Z_i(t)$ occurs at $t_i = \xi_{n-1,i} - \xi_{n,i}$; moreover, $Z_i(t)>0$ on the interval $(0,t_i)$. 

    We claim that it's sufficient to prove that
    \begin{equation}
        \label{eq:Ziplus1vsZi}
        Z_{i+1}(t)>Z_i(t),\qquad 0<t<t_i. 
    \end{equation}
    Indeed, suppose for a contradiction that $d_{i+1} < d_i$ for some $i=1,\dots,n-2$, i.e., that $t_{i+1}<t_i$.  Let $s=t_{i+1}$.  Then $Z_{i+1}(s)=0$. Since by assumption $s= t_{i+1}< t_i$, we have that $Z_{i}(s)>0 = Z_{i+1}(s)$, contradicting \eqref{eq:Ziplus1vsZi}.

    To prove \eqref{eq:Ziplus1vsZi}, let $$W_i(t) = Z_{i+1}(t) - Z_i(t).$$  Since $$Z_i(t) = \frac{1}{t} + \sum_{j\neq i} \frac{1}{\xi_{n,i} + t - \xi_{n,j} },$$ we have that $$\lim_{t\downarrow 0} \left(Z_{i+1}(s) - Z_i(s)\right)
    = \sum_{j\neq i+1} \frac{1}{\xi_{n,i+1} - \xi_{n,j}} - \sum_{j\neq i} \frac{1}{\xi_{n,i} - \xi_{n,j}}  
    = \xi_{n,i+1}-\xi_{n,i}$$ by part (\ref{fact:xiRootStietjes}) of the Lemma. Since the roots are listed in a monotone order, it follows that $$\lim_{t\downarrow 0} \left(Z_{i+1}(t) - Z_i(t)\right) \geq 0,$$ and so $W_i(t)\geq 0$ for all sufficiently small $t>0$.

    The Ricatti equation \eqref{eqn:RiccatiForR} and the definition of $Z_i$ implies $$Z'_i(t) = - Z_i(t)^2 + 2 (\xi_{n,i} +t) Z_i(t)-2n.$$ Since $W_i(t) = Z_{i+1} (t) - Z+i(t)$, we get \begin{eqnarray*}
        W'_i(t) &=& - (Z_{i+1}(t)^2 - Z_{i}(t)^2) + 2 (\xi_{n,i+1} + t) Z_{i+1}(t) - 2 (\xi_{n,i} + t) Z_{i}(t) \\
        &=& [ -Z_{i+1}(t) - Z_{i}(t) + 2 (\xi_{n,i+1} + t) ] W_{i}(t) + 2(\xi_{n,i+1}-\xi_{n,i}) Z_i(t).
    \end{eqnarray*}
    If $W_i(t)$ is not positive on $(0,t_i)$, then there must exist some $t_0$, $0<t_0<t_i$, so that $W_i(t)>0$ for all $0<t<t_0$ and $W_i(t_0)=0$.  Thus $W_i'(t_0)\leq 0$.   Substituting this into the equation for $W'_i(t)$ above gives
    $$
    W_i'(t_0) =  2(\xi_{n,i+1}-\xi_{n,i}) Z_i(t_0) > 0$$ since $t_0<t_i$, $Z_i(t)>0$ on $(0,t_i)$, and the roots are increasing.  But this contradicts $W_i'(t_0)\leq 0$. 
\end{proof}

\begin{lemma} \label{lem:ESisSumAndB}
    Suppose that $z_1 < z_2 \dots < z_N$ are $N$ points that satisfy $$ \sum_{j\neq i} \frac{1}{z_i-z_j} = cz_i$$ for all $i=1,\dots,N$. 
    
    For $S\subset\{1,\dots,N\}$ put $$E_z(S) = \sum_{i\in S,j\notin S}\frac{1}{|z_i -z_j|},\qquad B_z(S)=\sum_{i\in S,j\notin S,i>j} \frac{1}{z_i-z_j}.$$ Then $$ E_z(S) = -c \sum_{i\in S} z_i + 2 B_z(S).$$ 
\end{lemma}
\begin{proof}
    Notice that by symmetry $\sum_{i,j\in S,i\neq j} (z_i-z_j)^{-1} = 0$. Moreover, if $i>j$, then $z_i-z_j>0$, and if $i<j$, then $z_i-z_j <0$. Using this, we get:
    \begin{eqnarray*} 
        E_z(S) &=& \sum_{i\in S,j\notin S,i<j}\frac{1}{|z_i-z_j|}  + \sum_{i\in S,j\notin S,i>j}\frac{1}{|z_i-z_j|} \\ &=& - \sum_{i\in S,j\notin S}\frac{1}{z_i-z_j} + 2\sum_{i\in S,j\notin S,i>j}\frac{1}{z_i-z_j} \\ &=& -\sum_{i\in S,j\notin S}\frac{1}{z_i-z_j} - \underbrace{\sum_{i,j \in S, i\neq j}\frac{1}{z_i-z_j}}_{0} + 2B_z(S) \\
        &=& -\sum_{i\in S} \sum_{j\neq i} \frac{1}{z_i-z_j} + 2B_z(S) \\
        &=& - c\sum_{i\in S} z_i + 2B_z(S),
    \end{eqnarray*} where in the next to last line we used the hypothesis of the Lemma.
\end{proof}
\begin{lemma} \label{lem:energyVsCrossover}
    Keeping the notations of Lemma~\ref{lem:ESisSumAndB}, assume that $S\subset \{1,\dots,N\}$, and let $m=|S|$.  Put $I_m=\{1,\dots,m\}$. Then $$E_z(S)-E_z(I_m) = 2  B_z(S) - c \left(\sum_{i\in S}z_i - \sum_{i=1}^m z_i\right).$$
\end{lemma}

\begin{proof}
    Since if $i\in I_m$, $j\notin I_m$, $i<j$, $B_z(I_m)=0$.  
    By Lemma~\ref{lem:ESisSumAndB}, $$E_z(I_m) = -c \sum_{i=1}^m z_i + 2 B_z(I_m) = -c \sum_{i=1}^m z_i.$$
    Using this and the same Lemma gives $$E_z(S)-E(I_m) = 2B(S) - c \sum_{i\in S} z_i + c \sum_{i=1}^m z_i,$$
    which is the claimed expression.
\end{proof}

Let now $x_1<x_2<\dots<x_n$ be the rescaled roots of $H_n$, $x_i = (2n)^{-1/2} \xi_{n,i}$, and let $y_1<y_2<\dots,y_{n-1}$ be the roots of $H_{n-1}$ rescaled by the same constant, $y_i = (2n)^{-1/2} \xi_{n-1,i}$.  Then both sets $x_1<\cdots<x_n$ and $y_1<\cdots<y_{n-1}$ satisfy the hypothesis of Lemma~\ref{lem:ESisSumAndB} with $c=2n$. For $S\subset\{1,\dots,n\}$ and $S'\subset\{1,\dots,n-1\}$ we use the notation of Lemma~\ref{lem:ESisSumAndB}: $E_x(S) = \sum_{i\in S, j\in I_n\setminus S} |x_i-x_j|^-1$, $E_y(S) = \sum_{i\in S', j\in I_{n-1}\setminus S'} |y_i-y_j|^{-1}$. 

For $S\subset\{1,\dots,n\}$ let $T_L(S) = S\cap \{1,\dots,n-1\} \subset \{1,\dots,n-1\}$.  

\begin{lemma} \label{lem:ExvsEy}
Keeping the above notation, let $m=|S|$, $m'=|T_L(S)|$.  Then  $$E_x(S)-E_x(I_m) \geq E_y(T_L(S))-E_y(I_{m'}).$$  
\end{lemma}
\begin{proof}
    Suppose first that $n\notin S$, i.e.,  $T_L(S) = S$, so $m=m'$.  Let $d_i = y_i-x_i$. Write $S=\{s_1,\dots,s_m\}$ where $s_1<\dots<s_m$.  By Lemma~\ref{lem:hermiteFacts}.\ref{fact:diMonotone}, $d_1<d_2<\cdots<d_{n-1}$.  Therefore, using Lemma~\ref{lem:energyVsCrossover},
    \begin{gather*}
        E_x(S)-E_x(I_m) - E_y(S) + E_y(I_m) = 2B_x(S)-2B_y(S)  \\
             + 2n\left( \sum_{i\in S } x_i - \sum_{i=1}^m x_i - \sum_{i\in S} y_i + \sum_{i=1}^m x_i  \right) \\
            = 2B_x(S)-2B_y(S) + 2n\left(\sum_{i\in S} d_i - \sum_{i=1}^m d_i \right) \\
            = 2B_x(S)-2B_y(S) + 2n\left(\sum_{i=1}^m d_{s_i} - \sum_{i=1}^m d_i \right) \\
            \geq 2B_x(S)-2B_y(S).
    \end{gather*}
    since $s_i\geq i$ and thus $d_{s_i}\geq d_i$. Now, if $i\in S$, $j\notin S$, $i>j$, then $j\in \{1,\dots,n-1\}$.  Moreover, $$y_i-y_j =x_i + d_i -x_j -d_j = x_i -x_j + (d_i-d_j)\geq x_i-x_j$$ and so
    $$B_x(S)-B_y(S)=\sum_{i\in S}\sum_{j\in S,i>j} \frac{1}{x_i-x_j} - \frac{1}{y_i-y_j}\geq 0. $$ Thus $E_x(S)-E_x(I_m) - E_y(S) + E_y(I_m)\geq 0$ as claimed.

    Suppose now that $n\in S$. Let $S^c = I_n \setminus S$; then $n\notin S^c$.  By symmetry, $$E_x(S)=E_x(S^c).$$ Moreover, $$T_L(S^c) = \{1,\dots,n-1\}\setminus T_L(S)$$ and so $$E_y(T_L(S^c))=E_y(T_L(S)).$$  In this case, $m'=m-1$. By symmetry again, $$E_x(I_m) = E_x(I_{n-m}),\quad  E_y(I_{m-1})=E_y(I_{(n-1)-(m-1)})=E_y(I_{n-m}). $$ Putting these inequalities together yields the statement of the Lemma.
\end{proof}

\begin{theorem} \label{thrm:HermiteIntervalOptimal}
Let $S\subset\{1,\dots,n\}$, $m=|S|$. Then $$\sum_{i\in S,j\notin S} \frac{1}{|x_{n,i}-x_{n,j}|}
\geq \sum_{i\in I_m,j\notin I_m} \frac{1}{|x_{n,i}-x_{n,j}|},$$ where as before $x_{n,i} = (2n)^{-1/2}\xi_{n,i}$ are the scaled roots of the $n$-th Hermite polynomial $H_n$.
\end{theorem}

\begin{proof}
    The proof is by induction on $n$. If $n=1$, the statement is clear, since then $S=I_m$. 
    
    Assume now that the statement of the Theorem is true for $n-1$. Then by Lemma~\ref{lem:ExvsEy}, 
    \begin{multline*}
    \sum_{i\in S,j\notin S} \frac{1}{|x_{n,i}-x_{n,j}|} -\sum_{i\in I_m,j\notin I_m} \frac{1}{|x_{n,i}-x_{n,j}|}  =   
    \left(E_x(S)-E_x(I_m)\right) \\ \geq  \left(E_y(T_L(S))-E_y(I_{m'})\right),
    \end{multline*} where $m'=|T_L(S)|$. But $y_1<\dots <y_{n-1}$ are exactly the roots of $H_{n-1}$ rescaled by ${2n}^{-1/2}$, thus with $C = (n/(n-1))^{1/2}$, 
    \begin{multline*} \sum_{i\in S,j\notin S} \frac{1}{|x_{n,i}-x_{n,j}|} -\sum_{i\in I_m,j\notin I_m} \frac{1}{|x_{n,i}-x_{n,j}|} \\ 
    \geq C\left[\sum_{i\in T_L(S),j\notin T_L(S)} \frac{1}{|x_{n-1,i}-x_{n-1,j}|} -\sum_{i\in I_{m'},j\notin I_{m'}} \frac{1}{|x_{n-1,i}-x_{n-1,j}|}\right] \geq 0
    \end{multline*}
    by the induction hypothesis. This completes the proof.
\end{proof}
\bibliographystyle{amsalpha}
\bibliography{refs}

@article{Voiculescu1998,
  author = {Voiculescu, D.},
  title = {The analogues of entropy and of {Fisher}'s information measure in free probability theory {V}},
  journal = {Invent. Math.},
  volume = {132},
  pages = {189--227},
  year = {1998}
}

@book{NicaSpeicher,
  author = {Nica, A. and Speicher, R.},
  title = {Lectures on the Combinatorics of Free Probability},
  publisher = {Cambridge University Press},
  year = {2006}
}

@article{IsmailMuldoon1991,
  author = {Ismail, M. and Muldoon, M.},
  title = {A discrete approach to monotonicity of zeros of orthogonal polynomials},
  journal = {Transactions of the American Mathematical Society},
  volume = {323},
  number = {1},
  pages = {65--78},
  year = {1991}
}

@book{Szego,
  author = {Szegő, G.},
  title = {Orthogonal Polynomials},
  series = {American Mathematical Society Colloquium Publications},
  volume = {23},
  edition = {4},
  publisher = {American Mathematical Society},
  year = {1975}
}

@article{BakryLedoux1996,
  author = {Bakry, D. and Ledoux, M.},
  title = {{Lévy--Gromov}'s isoperimetric inequality for an infinite-dimensional diffusion generator},
  journal = {Inventiones Mathematicae},
  volume = {123},
  pages = {259--281},
  year = {1996}
}

@article{Bobkov1997,
  author = {Bobkov, S. G.},
  title = {An isoperimetric inequality on the discrete cube, and an elementary proof of the isoperimetric inequality in {Gauss} space},
  journal = {Annals of Probability},
  volume = {25},
  number = {1},
  pages = {206--214},
  year = {1997}
}

@book{VoiculescuDykemaNica,
    author={Voiculescu, D. and Dykema K. and Nica, A.},
    title ={Free random variables}, 
    series={CRM Monograph Series No. 1},
    publisher = {American Mathematical Society},
    year = {1992}
}

@article{BartheMaurey,
    author={Barthe, F. and  Maurey, B.},
    title={ Some remarks on isoperimetry of gaussian type},
    journal={Annales de l’I.H.P., section B},
    volume={36}, number={4},
    year={2000}, 
    pages={419--434}
}

@article{CapitaineHsuLedoux,
    author={M. Capitaine and E.P. Hsu and M. Ledoux},
    title={Martingale representation and a simple proof of
logarithmic {Sobolev} inequalities on path spaces}, 
    journal={Elect. Comm. in Probab.},
    volume={2},
    year={1997},
    pages={71--81}
}

@misc{githublink,
  author = {D. Shlyakhtenko},
  title = {Lean formalization of Free Bobkov's Inequality},
  howpublished = "\url{https://github.com/shlyakhtenko/free-bobkov-hermite}",
  year = {2026}, 
}

@article{MarcusSpeilmanSrivastava,
    author={Marcus, A. and Spielman, D. and Srivastava, N.},
    title={Finite free convolutions of polynomials},
    journal={Probab. Theory Relat. Fields},
    volume=182, 
    pages={807--848},
    year=2022}
\end{document}